\documentclass[a4paper,12pt,oneside]{article}
\usepackage{amssymb}
\usepackage{amsmath}
\usepackage{amsthm}
\usepackage{amsfonts}
\theoremstyle{plain}
\newtheorem{thm}{Theorem}[section]
\newtheorem{assumption}{Assumption}[section]
\newtheorem{lem}[thm]{Lemma}

\newtheorem{Def}{Definition}[section]

\newtheorem{exmp}{Example}[section]
\newtheorem{rem}{Remark}

\usepackage{setspace}
\usepackage{algorithm}
\usepackage{algpseudocode}
\usepackage{algorithmicx}
\usepackage{algpseudocode}
\usepackage{authblk}
\usepackage{cite}
\usepackage{enumerate}
\usepackage{enumitem}
\usepackage{url}
\usepackage{graphicx}
\graphicspath{{figures/}}
\usepackage{subfigure}
\usepackage{float}
\usepackage{enumerate}
\usepackage{diagbox}
\usepackage{tabularx}
\usepackage{tabu}
\usepackage{bm}
\usepackage{indentfirst}
\usepackage{geometry}
\title{Numerical construction of spherical $t$-designs by Barzilai-Borwein method}

\author[a]{Yuchen Xiao\thanks{E-mail address: fkxych@stu.jnu.edu.cn}}
\author[b,c]{Congpei An\thanks{Corresponding author: andbachcp@gmail.com, cpan@cuhk.edu.hk}}
\affil[a]{Department of Mathematics, Jinan University, Guangzhou, China}
\affil[b]{School of Economics Mathematics, Southwestern University of Finance and Economics, Chengdu, China}
\affil[c]{Current address: Department of Systems Engineering and Engineering Management, The Chinese University of Hong Kong, Hong Kong, China}

\date{}
\begin{document}
\maketitle
\begin{abstract}
\noindent A point set $\mathrm X_N$ on the unit sphere is a spherical $t$-design is equivalent to the nonnegative quantity $A_{N,t+1}$ vanished. We show that if $\mathrm X_N$ is a stationary point set of  $A_{N,t+1}$ and the minimal singular value of basis matrix is positive, then $\mathrm X_N$ is a spherical $t$-design. Moreover, the numerical construction of spherical $t$-designs is valid by using Barzilai-Borwein method. We obtain numerical spherical $t$-designs with $t+1$ up to $127$ at $N=(t+2)^2$.
\end{abstract}
%%---Main body -----
\section{Introduction}
Distributing finite points on the unit sphere is a challenging problem in the 21st century \cite{smale1998mathematical}. Spherical $t$-design is to find the `good' finite sets of points on the unit sphere $\mathbb{S}^d:=\{\bm x\in\mathbb{R}^{d+1}\vert\lVert\bm x\rVert_2=1\}$ for spherical polynomial approximations. Spherical $t$-design is very useful in approximation theory, geometry and combinatorics. Recently, it has been applied in quantum mechanics (for quantum $t$-design) and statistics (for rotatable design).
\begin{Def}
\label{Def1}
A finite set $\mathrm X_N:=\{\bm x_1,\ldots,\bm x_N\}\subset\mathbb S^d$ is a spherical $t$-design if for any polynomial $p:\mathbb{R}^{d+1}\to\mathbb{R}$ of degree at most $t$ such that the average value of $p$ on the $\mathrm X_N$ equals the average value of $p$ on $\mathbb{S}^d$, i.e.,
\begin{align}
\frac{1}{N}\sum_{i=1}^N p(\bm x_i)=\frac{1}{\vert\mathbb{S}^d\vert}\int_{\mathbb{S}^d} p(\bm x)\,d\omega(\bm x)\qquad\forall p\in\Pi_t,
\end{align}
where $\vert\mathbb{S}^d\vert$ is the surface of the whole unit sphere $\mathbb{S}^d$, $\Pi_t:=\Pi_t(\mathbb S^d)$ is the space of spherical polynomials on $\mathbb S^d$ with degree at most $t$ and $d\omega(\bm x)$ denotes the surface measure on $\mathbb{S}^d$.
\end{Def}
The concept of spherical $t$-design was introduced by Delsarte et al. \cite{delsarte1977spherical} in 1977. From then on, spherical $t$-designs have been studied extensively \cite{bondarenko2013optimal,sloan2004extremal,an2010well,chen2006existence,chen2011computational,sloan2009variational,bannai2009survey}.
In this paper, we pay attention to 2-dimensional unit sphere $\mathbb{S}^2$.

A lower bound on the number of points $N$ to construct a spherical $t$-design for any $t\ge 1$ on $\mathbb{S}^2$ was given in \cite{delsarte1977spherical}:
\begin{align*}
N\ge N^{*}=
\begin{cases}
\frac{1}{4}(t+1)(t+3),&\text{$t$ is odd,}\\
\frac{1}{4}(t+2)^2,&\text{$t$ is even.}
\end{cases}
\end{align*}
It is shown that the lower bound can not be achieved, in other words, there is no spherical $t$-design with $N^{*}$ points for any $t\ge 2$. Bondarenko et al. \cite{bondarenko2013optimal} proved spherical $t$-designs exists for $O(t^2)$ points. From the work of Chen et al. \cite{chen2011computational}, we know that spherical $t$-designs with $(t+1)^2$ points exist for all degrees $t$ up to 100 on $\mathbb{S}^2$. This encourage us to find higher degrees $t$ for spherical $t$-designs.

Extremal points are sets of $(t+1)^2$ points on $\mathbb{S}^2$ which maximize the determinant of a basis matrix for an arbitrary basis of $\Pi_t$ \cite{sloan2004extremal}. For $N=(t+1)^2$, Chen and Womersley verified a spherical $t$-design exist in a neighborhood of an extremal system \cite{chen2006existence}. For $N\ge{(t+1)^2}$, An et al. \cite{an2010well} verified extremal spherical $t$-designs exist for all degrees $t$ up to 60 and provided well conditioned spherical $t$-designs for interpolation and numerical integration.

By now, numerical methods have been developed for finding spherical $t$-designs. The problem of finding a spherical $t$-design is expressed as solving nonlinear equations or optimization problems \cite{sloan2009variational,an2010well}.
However, the first order methods for computing spherical $t$-designs are rarely developed. In this paper, we numerically construct spherical $t$-designs by using Barzilai-Borwein method (BB method). The BB method \cite{barzilai1988two} is a gradient method with modified step sizes, which is motivated by Newton's method but not involves any Hessian. Further investigations \cite{dai2002r} showed that BB method is locally $R$-linear convergent for general objective functions.

In the next section, we present the required techniques, definitions and first order conditions for spherical $t$-designs. The BB method for computing spherical $t$-designs and its convergence analysis are presented in Section \ref{sec3}. Numerical results for point sets which $t+1$ up to 127 and $N=(t+2)^2=16384$ are included in Section \ref{sec5}. Section \ref{sec6} ends this paper with a brief conclusion.
\section{First order conditions for spherical $t$-design}
\label{sec2}
$\{\mathrm Y_{0}^1,\mathrm Y_{1}^1,\ldots,\mathrm Y_{t}^{2t+1}\}$ for degree $n= 0,...,t$ and order $k= 1,\ldots,2n+1$ is a complete set of orthonormal real spherical harmonics basis for $\Pi_t$, where orthogonality with respect to the $L_2$ inner product \cite{freeden1998constructive},
\begin{align}
\langle f,g\rangle_{\mathbb S^2}:=\int_{\mathbb{S}^2} f(\bm x)g(\bm x)\,d\omega(\bm x),\qquad f,g\in L_2(\mathbb S^2).
\end{align}
Note that $\mathrm Y_{0}^1=\frac{1}{\sqrt{4\pi}}$. It is well known that the addition theorem for spherical harmonics on $\mathbb S^2$ gives
\begin{align}
\label{eq3}
\sum_{k=1}^{2n+1} \mathrm Y_{n}^k(\bm x)\mathrm Y_{n}^k(\bm y)=\frac{2n+1}{4\pi}\mathrm P_n(\langle\bm x,\bm y\rangle)\qquad\forall\bm x,\bm y\in\mathbb S^2,
\end{align}
where $\mathrm P_n:[-1,1]\to\mathbb R$ is Legendre polynomial and $\langle\bm x,\bm y\rangle:=\bm x^\top\bm y$ is the inner product in $\mathbb R^3$. Sloan and Womersley \cite{sloan2009variational} introduced a variational characterization of spherical $t$-designs
\begin{align}
\label{eq4}
A_{N,t}(\mathrm X_N):=\frac{4\pi}{N^2}\sum_{n=1}^t\sum_{k=1}^{2n+1}(\sum_{i=1}^N \mathrm Y_{n}^k(\bm x_i))^2=\frac{4\pi}{N^2}\sum_{j=1}^N\sum_{i=1}^N\sum_{n=1}^t\frac{2n+1}{4\pi}\mathrm P_n(\langle\bm x_j,\bm x_i\rangle).
\end{align}
\begin{thm}[\cite{sloan2009variational}]
\label{thm1}
Let $t\ge 1$, and $\mathrm X_N\subset\mathbb{S}^2$. Then
\begin{align}
0\leq A_{N,t}(\mathrm X_N)\leq (t+1)^2-1,
\end{align}
and $\mathrm X_N$ is a spherical {\it t-}design if and only if
\begin{align*}
A_{N,t}(\mathrm X_N)=0.
\end{align*}
\end{thm}
\noindent It is known that $\mathrm X_N$ is a spherical $t$-design if and only if $A_{N,t}(\mathrm X_N)$ vanished. Naturally, one might consider the first order condition to check the global minimizer of $A_{N,t}(\mathrm X_N)$.
\begin{Def}
\label{Def3}
A point $\bm x$ is a stationary point of $f\in C^1(\mathbb S^2)$ if $\nabla^*f(\bm x)=0$, where $\nabla^*:=\nabla^*_{\mathbb S^2}$ is the spherical gradient (or surface grident \cite{freeden1998constructive}) of $f$.
\end{Def}

Let the basis matrix be
$
\mathbf Y_t(\mathrm X_N):=
\begin{bmatrix}
\frac{1}{\sqrt{4\pi}}\mathbf e^\top\\\\
\mathbf Y_t^0(\mathrm X_N)
\end{bmatrix}\in\mathbb R^{(t+1)^2\times N},
$
where $\mathbf e=
\begin{bmatrix}
1\\
\vdots\\
1
\end{bmatrix}
\in\mathbb R^N$ and $\mathbf Y_t^0(\mathrm X_N)=
\begin{bmatrix}
\mathrm Y_1^1(\bm x_1)&\cdots&\mathrm Y_1^1(\bm x_N)\\
\vdots&\ddots&\vdots\\
\mathrm Y_t^{2t+1}(\bm x_1)&\cdots&\mathrm Y_t^{2t+1}(\bm x_N)
\end{bmatrix}
\in\mathbb R^{(t+1)^2-1\times N}.$

\begin{Def}
\label{Def6}
A finite set $\mathrm X_N:=\{\bm x_1,\ldots,\bm x_N\}\subset\mathbb S^2$ is called a fundamental system for $\Pi_t$ if the zero polynomial is the only element of $\Pi_t$ that vanishes at each point in $\mathrm X_N$.
\end{Def}
An et al. \cite{an2010well} described the fundamental system in finding spherical $t$-designs.
\begin{lem}[\cite{an2010well}]
\label{lem4}
A set $\mathrm X_N\subset\mathbb S^2$ is a fundamental system for $\Pi_t$ if and only if $\mathbf Y_t(\mathrm X_N)$ is of full row rank $(t+1)^2$.
\end{lem}
\begin{lem}[\cite{an2010well}]
\label{lem2}
Let $t\ge 2$ and $N\ge (t+2)^2$ . Assume $\mathrm X_N\subset\mathbb S^2$ is a stationary
point set of $A_{N,t}$ and $\mathrm X_N$ is a fundamental system for $\Pi_{t+1}$. Then $\mathrm X_N$ is a spherical $t$-design.
\end{lem}
\noindent Based on these results, we have the applicable first order condition for spherical $t$-designs as follows.
\begin{thm}
\label{thm2}
Let $t\ge 2$ and $N\ge(t+2)^2$. Assume $\mathrm X_N\subset\mathbb S^2$ is a stationary point set of $A_{N,t}$ and the minimal singular value of basis matrix $\mathbf Y_{t+1}(\mathrm X_N)$ is positive. Then $\mathrm X_N$ is a spherical $t$-design.
\end{thm}
\begin{proof}
Suppose that the minimal singular value of $\mathbf Y_{t+1}(\mathrm X_N)$ is positive, then we have all the singular values of $\mathbf Y_{t+1}(\mathrm X_N)$ are positive immediately. We know that the number of non-zero singular values of $\mathbf Y_{t+1}(\mathrm X_N)$ equals the rank of $\mathbf Y_{t+1}(\mathrm X_N)$, so $\mathbf Y_{t+1}(\mathrm X_N)$ is of full rank, which means $\mathrm X_N$ is a fundamental system of $\Pi_{t+1}$ by Lemma \ref{lem4}. And then suppose that $\mathrm X_N$ is a stationary point set, then $\mathrm X_N$ is a spherical $t$-design by Lemma \ref{lem2}. Hence, we complete the proof.
\end{proof}
\noindent Theorem \ref{thm2} is useful in first order optimization method, which provides a simple way to verify the global minimizer to the objective function.
\section{Iterative methods for finding spherical $t$-designs}
\label{sec3}
\subsection{Algorithm design}
Fix $N$ and $t$, for objective function $A_{N,t}: \mathbb S^{2\times N}\to\mathbb R$, we consider the optimization problem
\begin{align}
\label{eq5}
\min_{\mathrm X_N\subset\mathbb S^2} A_{N,t}(\mathrm X_N).
\end{align}
Apparently, $A_{N,t}$ is a non-convex function. For computing $\mathrm X_N$ conveniently, we assume the first point $\bm x_1=(0,0,1)^\top$ is the north pole point and the second point $\bm x_2=(x_2,0,z_2)^\top$ is on the primer meridian. Then we can define coordinates convert functions $\eta:\mathbb R^{3\times N}\to\mathbb R^{1\times 2N-3}$ which can convert Cartesian coordinates into spherical coordinates as a vector, and $\mu:\mathbb R^{1\times 2N-3}\to\mathbb R^{3\times N}$ which can convert a vector form spherical coordinates into Cartesian coordinates as a matrix. So for $(\theta,\phi)\in[0,\pi]\times[0,2\pi)$ we have
\begin{align*}
\eta(\mathrm X_N)&=\eta
\begin{bmatrix}
\sin\theta_1\cos\phi_1 & \cdots & \sin\theta_N\cos\phi_N \\
\sin\theta_1\sin\phi_1 & \cdots & \sin\theta_N\sin\phi_N \\
\cos\theta_1 & \cdots & \cos\theta_N
\end{bmatrix}=(\Theta,\Phi),\\
\mu(\eta(\mathrm X_N))&= \mu(\Theta,\Phi)=\begin{bmatrix}
x_1 & \cdots & x_N \\
y_1 & \cdots & y_N \\
z_1 & \cdots & z_N
\end{bmatrix},
\end{align*} 
where vector $\Theta:=(\theta_2,\ldots,\theta_N)\in\mathbb R^{1\times N-1}$ and vector $\Phi:=(\phi_3,\ldots,\phi_N)\in\mathbb R^{1\times N-2}$.

We apply BB method in \cite{barzilai1988two} to construct Algorithm \ref{alg:BBs} for seeking an efficient way to compute $A_{N,t}$, that $x$ achieves the local minimum. Due to the universality of quasi-Newton method \cite{sloan2009variational}, we also apply quasi-Newton method for comparing the efficiency. And then we try to use Theorem \ref{thm2} to prove the local minimum we found is the global minimum, that is, we find the real numerical spherical $t$-design. 

To make sure that objective function $f(x_k)$ is sufficient to descend and approximate to $\varepsilon$ which is as near as $0$, we use Armijo-Goldstein rule \cite{sun2006optimization} and backtracking line search \cite{sun2006optimization} to lead BB method in a proper way to find local minimum $x^*$.

\begin{algorithm}[H]
  \caption{Barzilai-Borwein method for computing spherical $t$-designs}
  \label{alg:BBs}
  \begin{algorithmic}[1]
    \Require
      {$t$: spherical polynomial degree;
      $N$: number of points;
      $\mathrm X_N$: distributing $N$ points on unit sphere $\mathbb S^2$;
      $K_{max}$: maximum iterations;
      $\varepsilon_1$: termination tolerance on the first-order optimality;
      $\varepsilon_2$: termination tolerance on progress in terms of function or parameter changes.

     \setlength{\parindent}{1.5em} Initialize $k=1$, $x_0=x_1=\eta(\mathrm X_N)$, $f_0=f_1=A_{N,t}(\mu(x_0))$, $g_0=g_1=\eta(\nabla^*A_{N,t}(\mu(x_0)))$ and $\alpha_1=1$. }
    \While{$k\leq K_{max}$ and $\lVert g_{k+1}-g_{k}\rVert_2>\varepsilon_1$, $\lVert f_{k+1}-f_{k}\rVert_2>\varepsilon_2$ or $\lVert x_{k+1}-x_{k}\rVert_2>\varepsilon_2$}
      \State $s_k=x_{k}-x_{k-1}$, $y_k=g_{k}-g_{k-1}$
      \State  compute step size $\alpha_{k}=(s_k^\top s_k)(s_k^\top y_k)^{-1}$
      \If{$\alpha_k\leq 10^{-10}$ or $\alpha_k\geq 10^{10}$}
      \State $\alpha_k=1$
      \EndIf
      \If {$f(x_{k}-\alpha_{k}g_{k})\leq f(x_k)-\alpha_k\rho g_k^\top g_k$ and \\\ \quad$f(x_{k}-\alpha_{k}g_{k})\geq f(x_k)-\alpha_k(1-\rho)g_k^\top g_k$, $(0<\rho<\frac{1}{2})$}
      \State $\alpha_k=\alpha_k$ (Armijo-Goldstein rule)
      \Else \
      \State $\alpha_k=\tau\alpha_{k-1}$, $\tau\in(0,1)$ (backtracking line search)
      \EndIf
      \State $x_{k+1}=x_k-\alpha_k g_k$
      \State compute $f_{k+1}=A_{N,t}(\mu(x_{k+1}))$ and search direction $g_{k+1}=\eta(\nabla^* A_{N,t}(\mu(x_{k+1})))$
    \EndWhile
    \Ensure
      numerical spherical $t$-designs $x^*\subset\mathbb S^2$.
  \end{algorithmic}
\end{algorithm}
Now we give a small numerical example by using Algorithm \ref{alg:BBs}, which is used to illustrate the numerical construction of spherical $t$-design.
\begin{exmp}
\label{exmp1}
We generate spiral points $\mathrm X_4=\{\bm x_1,\bm x_2,\bm x_3,\bm x_4\}\subset\mathbb S^2$ from \cite{bauer2000distribution},
\begin{equation*}
\bm x_1=\begin{bmatrix}
0\\0\\1\end{bmatrix},
\bm x_2=\begin{bmatrix}
0.9872\\0\\-0.1595\end{bmatrix},
\bm x_3=\begin{bmatrix}
-0.3977\\0.6727\\-0.6239\end{bmatrix},
\bm x_4=\begin{bmatrix}
-0.6533\\-0.7455\\-0.1318\end{bmatrix}.
\end{equation*}
By using Algorithm \ref{alg:BBs}, we obtain the termination output: $k=25$, $\lvert A_{N,t}(\mathrm X_4^*)\rvert=2.775558\times 10^{-17}$, $\lVert \nabla^{*}A_{N,t}(\mathrm X_4^*)\rVert=1.0446\times 10^{-8}$ and $\mathrm X_4^*$ ends with value
\begin{equation*}
\bm x_1^*=\begin{bmatrix}
0\\0\\1\end{bmatrix},
\bm x_2^*=\begin{bmatrix}
0.9428\\0\\-0.3333\end{bmatrix},
\bm x_3^*=\begin{bmatrix}
-0.4714\\0.8165\\-0.3333\end{bmatrix},
\bm x_4^*=\begin{bmatrix}
-0.4714\\-0.8165\\-0.3333\end{bmatrix}.
\end{equation*}
In fact, $\mathrm X_4^*$ is a set of regular tetrahedron vertices, which is known as a spherical $2$-design. As a result, Algorithm \ref{alg:BBs} reaches the global minimum $\mathrm X_4^*$, thus numerical solutions for spherical $2$-design found. We can see the explicit change of $X_4$ by using Algorithm \ref{alg:BBs} from Figure \ref{Fig.lable1} and the behavior of objective function from Figure \ref{Fig.lable5}.
\begin{figure}[H] 
\centering 
\subfigure[Initial vertices]{ 
\label{Fig.sub.1} 
\includegraphics[width=0.4\textwidth]{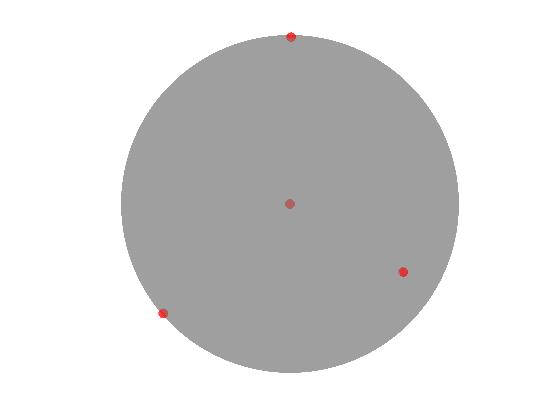}} 
\subfigure[Final vertices]{ 
\label{Fig.sub.2} 
\includegraphics[width=0.4\textwidth]{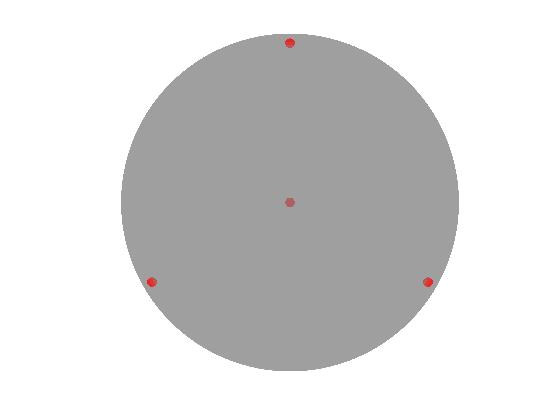}} 
\caption{Numerical simulation of regular tetrahedron vertices on $\mathbb S^2$ by using Algorithm \ref{alg:BBs}}
\label{Fig.lable1} 
\end{figure}
\begin{figure}[H] 
\centering 
\subfigure[The behavior of $A_{N,t}$]{ 
\label{Fig.sub.7} 
\includegraphics[width=0.4\textwidth]{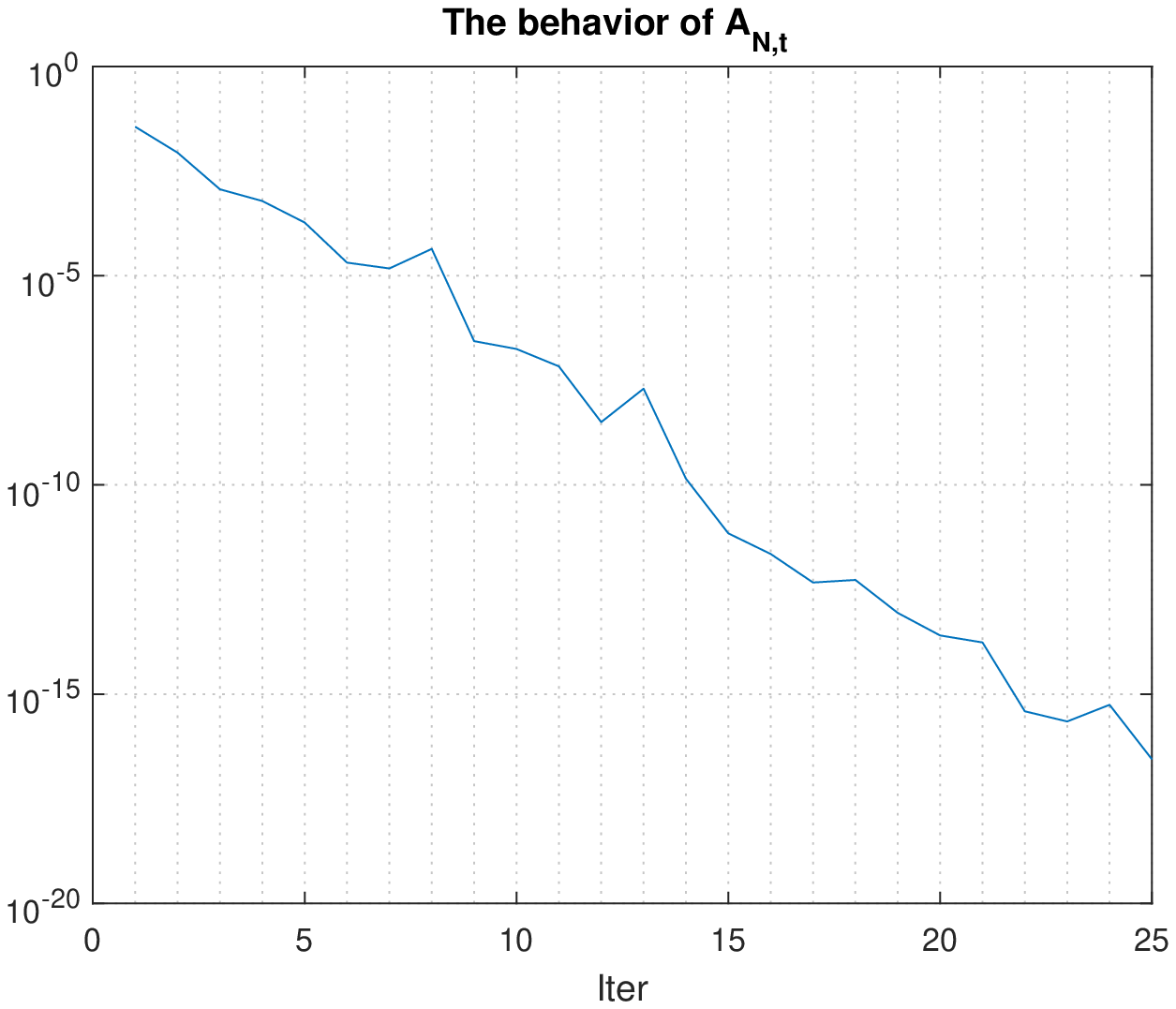}} 
\subfigure[The behavior of $\lVert\nabla^* A_{N,t}\rVert_2$]{ 
\label{Fig.sub.8} 
\includegraphics[width=0.4\textwidth]{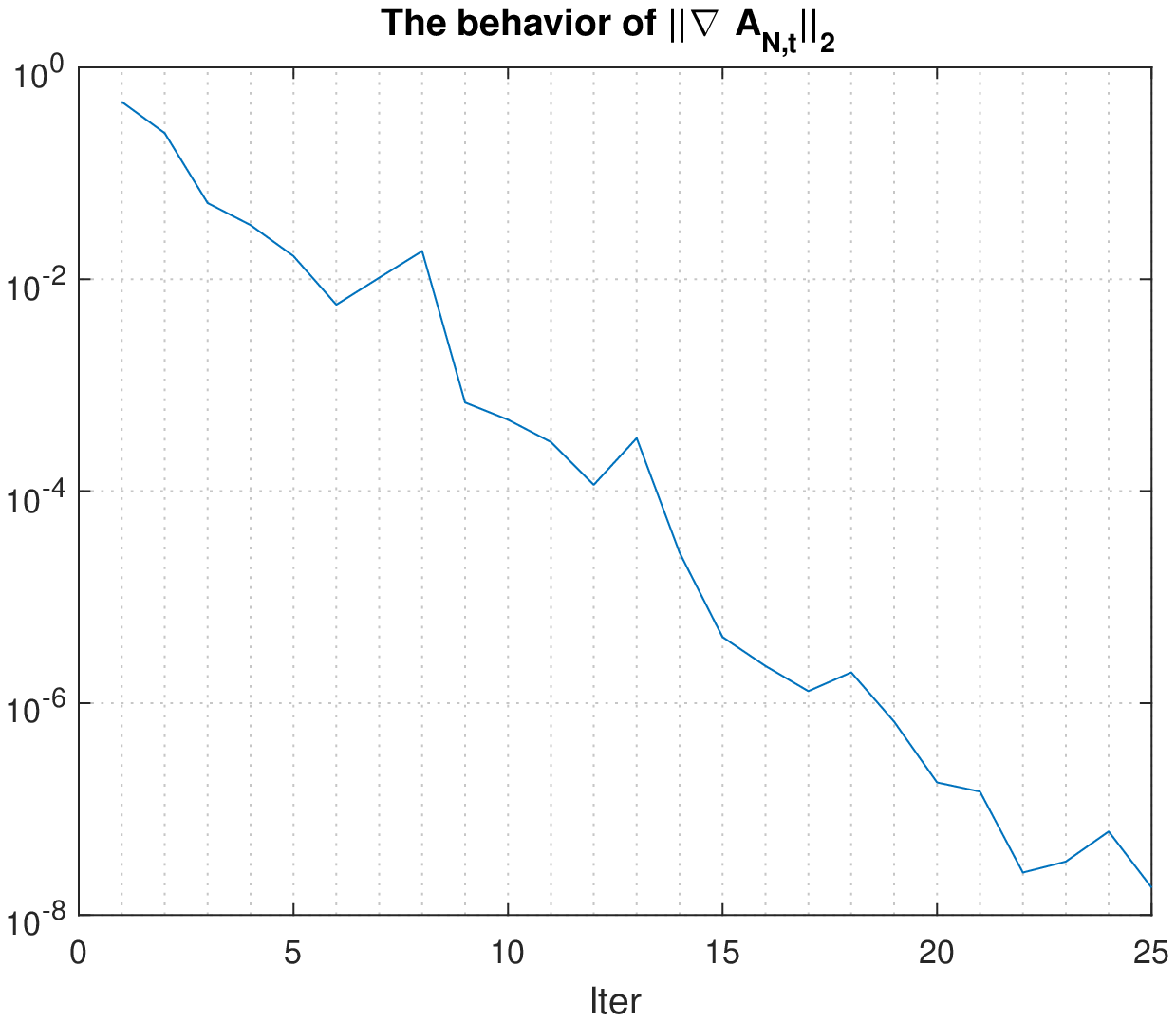}} 
\caption{Numerical behavior of $A_{N,t}$ and $\lVert\nabla^* A_{N,t}\rVert_2$ with $t=2,N=4$ by using Algorithm \ref{alg:BBs}}
\label{Fig.lable5} 
\end{figure}
\end{exmp}
\subsection{Convergence analysis}
From the view of (\ref{eq4}), we know $A_{N,t}\in C^t(\mathbb S^2)$ for $t\ge 2$. We assume that
\begin{assumption}
\label{asump1}
The level set $D:=\{x\in\mathbb R^n\vert f(x)\leq f(x_l)\}$ is bounded, and there exists $M>0$ such that $\lVert\nabla^2f(x)\rVert\leq M$, where $\nabla^2 f$ is a Hessian matrix of $f(x)$.
\end{assumption}
Now we present the convergence result on Algorithm \ref{alg:BBs}. We shall mention that the idea of proof originated in \cite{yuan1999problems,powell1976algorithm}.
\begin{thm}
\label{thm3}
Let $x_1=\eta(\mathrm X_N)$ be an initial point and $g_1=\eta(\nabla^*A_{N,t}(\mu(x_1)))$ and assume Assumption \ref{asump1} holds. Suppose that ${x_k}$ is generated by Algorithm \ref{alg:BBs}, then $\lim\limits_{k\to\infty}\inf{\lVert g_k\rVert}=0$.
\end{thm}
\begin{proof}
By using the Armijo rule (mark 8) from Algorithm \ref{alg:BBs} and mean value theorem, we have
\begin{align}
f(x_k)-f(x_k-\alpha_kg_k)=\alpha_k\nabla f(x_k-\kappa\alpha_kg_k)^\top g_k\leq \alpha_k(1-\rho)g_k^\top g_k,
\end{align}
where $\nabla f$ is a gradient of $f$ and $\kappa\in(0,1)$, then
\begin{align}
\rho g_k^\top g_k\leq (\nabla f(x_k)-\nabla f(x_k-\kappa\alpha_kg_k))^\top g_k.
\end{align}
According to Cauchy inequality, we obtain
\begin{align}
\label{eq1}
\rho g_k^\top g_k\leq (\nabla f(x_k)-\nabla f(x_k-\kappa\alpha_kg_k))^\top g_k\leq\lVert\nabla f(x_k)-\nabla f(x_k-\kappa\alpha_kg_k)\rVert\lVert g_k\rVert,
\end{align}
moreover, by using mean value theorem
\begin{align}
\label{eq2}
\lVert\nabla f(x_k)-\nabla f(x_k-\kappa\alpha_kg_k)\rVert=\lVert\int_0^1 F(x_k-\xi\kappa\alpha_kg_k)\kappa\alpha_kg_k\,d\xi\rVert\leq M\kappa\alpha_k\lVert g_k\rVert.
\end{align}
Combine (\ref{eq1}) and (\ref{eq2}), we know
\begin{align}
\rho g_k^\top g_k\leq M\kappa\alpha_k\lVert g_k\rVert,
\end{align}
therefore
\begin{align}
\label{eq6}
\alpha_k\lVert g_k\rVert\geq\frac{\rho g_k^\top g_k}{M\lVert g_k\rVert}.
\end{align}
By Armijo rule (mark 7) from Algorithm \ref{alg:BBs} and (\ref{eq6}), we have
\begin{align}
f(x_{k+1})\leq f(x_k)-\alpha_k\rho\lVert g_k\rVert\frac{g_k^\top g_k}{\lVert g_k\rVert}\leq f(x_k)-\frac{\rho^2}{M}(\frac{g_k^\top g_k}{\lVert g_k\rVert})^2,
\end{align}
thus
\begin{align}
\sum\limits_{j=1}^k(\frac{g_j^\top g_j}{\lVert g_j\rVert})^2\leq\frac{M}{\rho^2}(f(x_1)-f(x_{k+1})).
\end{align}
Since $D$ is bounded, we know $\lim\limits_{k\to\infty}f(x_{k+1})$ exists, therefore
\begin{align}
\label{eq7}
\sum\limits_{j=1}^k(\frac{g_j^\top g_j}{\lVert g_j\rVert})^2<+\infty,
\end{align}
hence
\begin{align}
\lim\limits_{k\to\infty}\frac{g_k^\top g_k}{\lVert g_k\rVert}=0.
\end{align}
Now we assume that $\lim\limits_{k\to\infty}\sup\lVert g_k\rVert\neq 0$. We can find a set of $\{k_n\}$ ($n\in\mathbb Z^+$), when $n\to\infty$, $k_n\to\infty$, and there exist $\epsilon>0$ such that $\lVert g_{k_n}\rVert>\epsilon$. Therefore, (\ref{eq7}) can not be hold, which contradicts. Thus, $\lim\limits_{k\to\infty}\inf\lVert g_k\rVert=\lim\limits_{k\to\infty}\sup\lVert g_k\rVert=0$, we complete the proof.
\end{proof}
Theorem \ref{thm3} shows the convergence of Algorithm \ref{alg:BBs}. Based on the above theorems, we summarize the following results. 
\begin{rem}
\label{rem1}
Let $x_k$ be the starting points set of Algorithm \ref{alg:BBs}, by Theorem \ref{thm3}, then there exist $\epsilon>0$ such that $\lim\limits_{k\to\infty}\nabla^* A_{N,t}(\mu(x_k))=\bm 0$ when $k>\epsilon$, where $\bm 0\in\mathbb R^{1\times N}$ is a zero vector. Therefore, $\mathrm X_N$ is a stationary points set of $A_{N,t}$.
\end{rem}
\begin{rem}
\label{rem2}
Let $x_k$ be the starting points set of Algorithm \ref{alg:BBs}, then we have $\lim\limits_{k\to\infty}A_{N,t}(\mu(x_k))=0$. If Theorem \ref{thm2} is established in $x^*$, then $x^*$ is a spherical $t$-design. 
\end{rem}
\section{Numerical results}
\label{sec5}
Based on the code in \cite{sloan2009variational,schmidt2005minFunc}, we present the feasibility of Algorithm {\ref{alg:BBs}} to compute spherical $t$-design with the point set $\mathrm X_N$ where $N=(t+2)^2$ for $t+1$ up to 127. As an initial point set $\mathrm X_N$ to solve the optimization problem of minimizing $A_{N,t}(\mathrm X_N)$ from (\ref{eq4}), we use the extremal systems from \cite{sloan2004extremal} without any additional constraints. To make sure BB method is meaningful in spherical $t$-designs, we compare BB method with quasi-Newton method(QN). These methods are implemented in Matlab R2015b and tested on an Intel Core i7 4710MQ CPU with 16 GB DDR3L memory and a 64 Bit Windows 10 Education.

We present the results in Table \ref{table1} and Table \ref{table2}, these numerical spherical $t$-designs can be founded in \cite{anxiao2019sphere}. We observe that BB method cost less time than quasi-Newton method, especially in large $\mathrm X_N$. Furthermore, all point sets $\mathrm X_N$ are verified to be fundamental systems. In fact, we use singular value decomposition (SVD) \cite{sun2006optimization} to obtain all singular values of $\mathbf Y_{t+1}(\mathrm X_N)$, which are defined as $\{\sigma_i\}$ for $i=1,\ldots,(t+2)^2$. As a result, the $\min(\sigma_i)>0$, then $\mathbf Y_{t+1}(\mathrm X_N)$ is of full rank, thus $\mathrm X_N$ is a fundamental system. This is a strong numerical support to Theorem \ref{thm2}. Here we set $\varepsilon_1=\varepsilon_2=10^{-16}$.

Figure \ref{Fig.sub.3} and Figure \ref{Fig.sub.5} are well exhibited the locally R-linear convergence \cite{dai2002r} of BB method by numerical computation of $A_{N,t}$ with $t=50,N=2601$. We can see that $A_{N,t}$ converges to $0$ with iteration increase. 

\begin{table}[H]
\centering
\caption{Computing of spherical $t$-designs by BB method}\label{table1}
\begin{tabular}{|cc|ccclc|}
\hline
 $t+1$ & $N$ & Iteration & $A_{N,t}(\mathrm X_N)$ & $\lVert\nabla^* A_{N,t}(\mathrm X_N)\rVert_{\infty}$ & Time & $\min(\sigma_i)$\\
\hline
10 & 121 & 100 & 7.796661e-16 & 1.8478e-09 & 1.049909s & 1.3270 \\
50 & 2601 & 335 & 1.879594e-12 & 6.3307e-09 & 336.663545s & 2.3394 \\
96 & 9409 & 715 & 8.237123e-10 & 7.3212e-08 & 22724.767736s & 2.1647 \\
127 & 16384 & 803 & 8.229142e-10 & 2.7122e-08 & 84579.358811s & 1.9673 \\
\hline
\end{tabular}
\end{table}

\begin{table}[H]
\centering
\caption{Computing of spherical $t$-designs by quasi-Newton method}\label{table2}
\begin{tabular}{|cc|ccclc|}
\hline
 $t+1$ & $N$ & Iteration & $A_{N,t}(\mathrm X_N)$ & $\lVert\nabla^* A_{N,t}(\mathrm X_N)\rVert_{\infty}$ & Time & $\min(\sigma_i)$\\
\hline
10 & 121 & 81 & 1.054716e-15 & 2.1856e-09 & 1.062004s & 1.3232 \\
50 & 2601 & 278  & 8.455657e-15 & 2.4398e-09 & 397.480675s & 2.3380 \\
96 & 9409  & 465 & 1.418436e-14 & 1.4049e-09 & 40101.809512s & 2.1637 \\
127 & 16384  & 543 & 3.709079e-13 & 2.0536e-09 & 143631.176093s & 1.9656 \\
\hline
\end{tabular}
\end{table}

\begin{figure}[H]
\centering
\subfigure[Barzilai-Borwein method]{
\label{Fig.sub.3}
\includegraphics[width=0.4\textwidth]{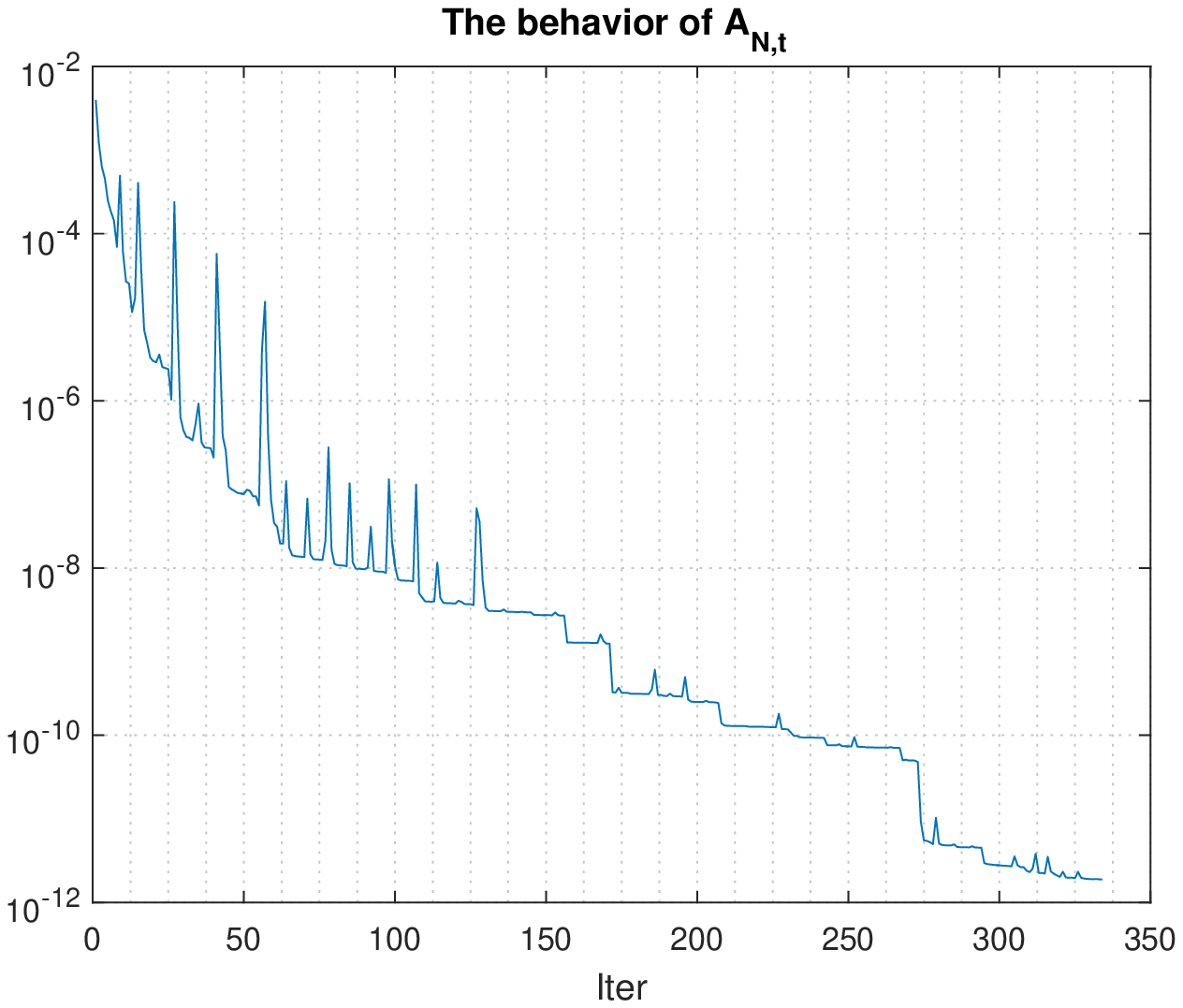}}
\subfigure[Quasi-Newton method]{
\label{Fig.sub.4}
\includegraphics[width=0.4\textwidth]{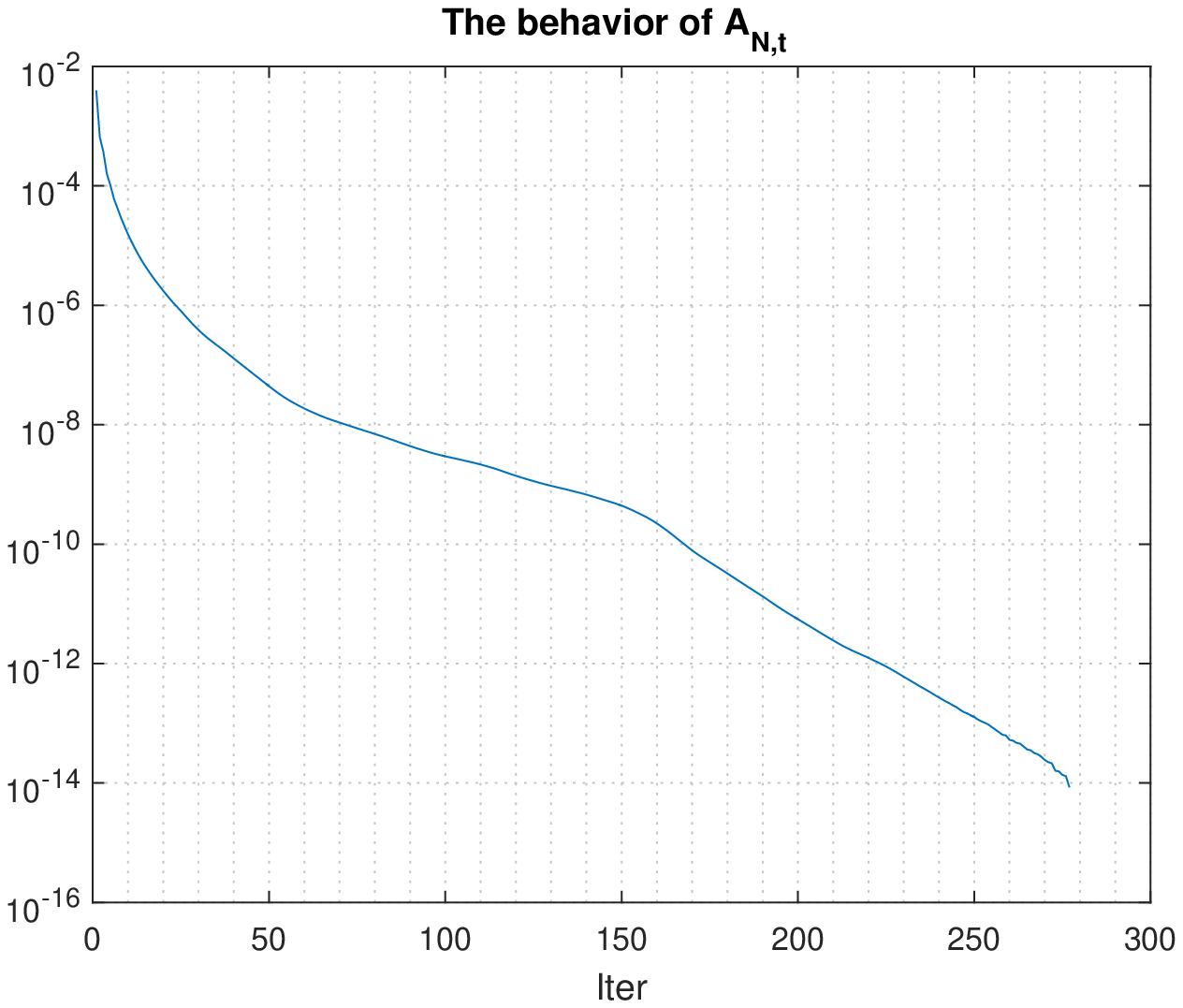}}
\caption{The behavior of $A_{N,t}$ for $t=50,N=(t+1)^2$ on $\mathbb S^2$ in each iteration}
\label{Fig.lable2}
\end{figure}

\begin{figure}[H]
\centering
\subfigure[Barzilai-Borwein method]{
\label{Fig.sub.5}
\includegraphics[width=0.4\textwidth]{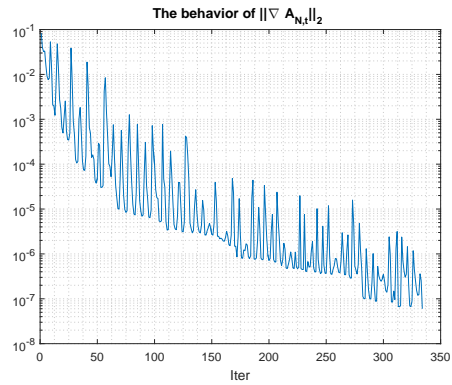}}
\subfigure[Quasi-Newton method]{
\label{Fig.sub.6}
\includegraphics[width=0.4\textwidth]{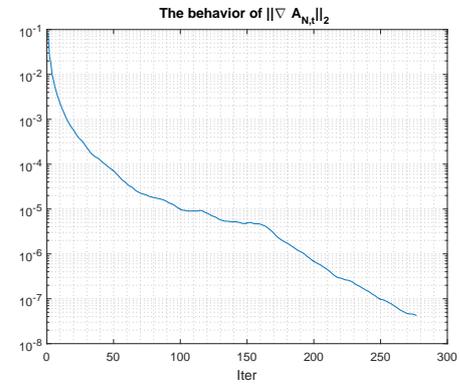}}
\caption{The behavior of $\lVert\nabla^* A_{N,t}(\mathrm X_N)\rVert_{2}$ for $t=50,N=(t+1)^2$ on $\mathbb S^2$ in each iteration}
\label{Fig.lable3}
\end{figure}

\begin{figure}[H] 
\centering 
\subfigure[Barzilai-Borwein method]{ 
\label{Fig.sub.9} 
\includegraphics[width=0.4\textwidth]{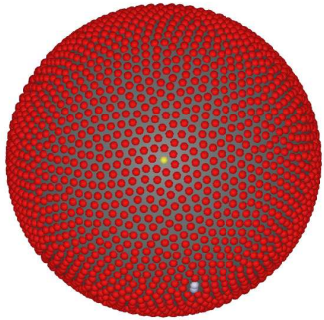}} 
\subfigure[Quasi-Newton method]{ 
\label{Fig.sub.10} 
\includegraphics[width=0.4\textwidth]{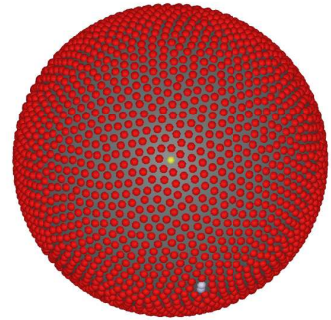}} 
\caption{Numerical simulation for $t=50,N=(t+1)^2$ on $\mathbb S^2$ by using different methods} 
\label{Fig.lable6} 
\end{figure}

\section{Conclusion}
\label{sec6}
In this paper, we employ Barzilai-Borwein method for finding numerical spherical $t$-designs with $t$ up to 126 with $N=(t+1)^2$. This method performs high efficiency and accuracy. Moreover, we check numerical solution as global minimizer with positivity of minimal singular value of basis matrix. Numerical experiments show that Barzilai-Borwein method is better than quasi-Newton method in time efficiency for solving large scale spherical $t$-designs. These numerical results are interesting and inspiring. The numerical construction of higher order spherical $t$-designs are expected in future study.
\section*{Acknowledgements}
This work is supported by National Natural Science Foundation of China (Grant No.11301222) and the Opening Project of Guangdong Province Key Laboratory of Computational Science at the Sun Yat-sen University (Grant No.2018014).

%\bibliographystyle{elsarticle-num}
%\bibliography{mybibtex}

\small 
\begin{thebibliography}{10}
\expandafter\ifx\csname url\endcsname\relax
  \def\url#1{\texttt{#1}}\fi
\expandafter\ifx\csname urlprefix\endcsname\relax\def\urlprefix{URL }\fi
\expandafter\ifx\csname href\endcsname\relax
  \def\href#1#2{#2} \def\path#1{#1}\fi

\bibitem{smale1998mathematical}
S.~Smale, Mathematical problems for the next century, The Mathematical
  Intelligencer 20~(2) (1998) 7--15.

\bibitem{delsarte1977spherical}
P.~Delsarte, J.~Goethals, J.~Seidel, Spherical codes and designs, Geometriae
  Dedicata 6~(3) (1977) 363--388.

\bibitem{bondarenko2013optimal}
A.~Bondarenko, D.~Radchenko, M.~Viazovska, Optimal asymptotic bounds for
  spherical designs, Annals of mathematics (2013) 443--452.

\bibitem{sloan2004extremal}
I.~H. Sloan, R.~S. Womersley, Extremal systems of points and numerical
  integration on the sphere, Advances in Computational Mathematics 21~(1-2)
  (2004) 107--125.

\bibitem{an2010well}
C.~An, X.~Chen, I.~H. Sloan, R.~S. Womersley, Well conditioned spherical
  designs for integration and interpolation on the two-sphere, SIAM Journal on
  Numerical Analysis 48~(6) (2010) 2135--2157.

\bibitem{chen2006existence}
X.~Chen, R.~S. Womersley, Existence of solutions to systems of underdetermined
  equations and spherical designs, SIAM Journal on Numerical Analysis 44~(6)
  (2006) 2326--2341.

\bibitem{chen2011computational}
X.~Chen, A.~Frommer, B.~Lang, Computational existence proofs for spherical
  t-designs, Numerische Mathematik 117~(2) (2011) 289--305.

\bibitem{sloan2009variational}
I.~H. Sloan, R.~S. Womersley, A variational characterisation of spherical
  designs, Journal of Approximation Theory 159~(2) (2009) 308--318.

\bibitem{bannai2009survey}
E.~Bannai, E.~Bannai, A survey on spherical designs and algebraic combinatorics
  on spheres, European Journal of Combinatorics 30~(6) (2009) 1392--1425.

\bibitem{barzilai1988two}
J.~Barzilai, J.~M. Borwein, Two-point step size gradient methods, IMA Journal
  of Numerical Analysis 8~(1) (1988) 141--148.

\bibitem{dai2002r}
Y.~Dai, L.~Liao, R-linear convergence of the barzilai and borwein gradient
  method, IMA Journal of Numerical Analysis 22~(1) (2002) 1--10.

\bibitem{freeden1998constructive}
W.~Freeden, T.~Gervens, M.~Schreiner, Constructive approximation on the sphere
  with applications to geomathematics, Oxford University Press on Demand, 1998.

\bibitem{sun2006optimization}
W.~Sun, Y.-X. Yuan, Optimization theory and methods: nonlinear programming,
  Vol.~1, Springer Science \& Business Media, 2006.

\bibitem{bauer2000distribution}
R.~Bauer, Distribution of points on a sphere with application to star catalogs,
  Journal of Guidance, Control, and Dynamics 23~(1) (2000) 130--137.

\bibitem{yuan1999problems}
Y.-X. Yuan, Problems on convergence of unconstrained optimization algorithms,
  Numerical Linear Algebra and Optimization (1999) 95--107.

\bibitem{powell1976algorithm}
M.~J.~D. Powell, Algorithm for minimization without exact line searches,
  Nonlinear programming 9 (1976) 53.

\bibitem{schmidt2005minFunc}
M.~Schmidt, minfunc: unconstrained differentiable multivariate optimization in
  matlab, \url{http://www.cs.ubc.ca/~schmidtm/Software/minFunc.html} (2005).

\bibitem{anxiao2019sphere}
Y.~Xiao, C.~An, Numerical results of spherical $t$-designs with $t$ up to 127,
  \url{https://github.com/fkxych/Numerical-construction-of-spherical-t-designs-by-Barzilai-Borwein-method.git}
  (2019).

\end{thebibliography}

\end{document}